\documentclass[12pt]{extarticle}
\usepackage{layout,  graphicx,  amssymb,amsmath,amsthm,bm,mathtools}
\usepackage{microtype}
\mathtoolsset{showonlyrefs}  
\usepackage[total={158mm,209mm}]{geometry} 
\usepackage[bitstream-charter]{mathdesign}
\usepackage[english]{babel}
\usepackage[utf8x]{inputenc}
\usepackage[colorinlistoftodos]{todonotes}
\usepackage[colorlinks=true, allcolors=blue,hyperindex,breaklinks]{hyperref}

\author{Johan Andersson\thanks{Email:johan.andersson@oru.se \, Address:Department of Mathematics, School of Science and Technology, {\"O}rebro University, {\"O}rebro, SE-701 82 Sweden. }}

\title{Mittag-Leffler type theorems  for Helson zeta-functions\thanks{The main ideas of this paper were concieved and a major part of the paper was written while attending the analytic number theory program at  institut Mittag-Leffler in January and March 2024.}}
\theoremstyle{plain}
\newtheorem{thm}{Theorem}
\newtheorem{lem}{Lemma}%
\newtheorem{cor}{Corollary}
\theoremstyle{definition}

\date{}
\DeclareOldFontCommand{\rm}{\normalfont\rmfamily}{\mathrm}
\DeclareOldFontCommand{\sf}{\normalfont\sffamily}{\mathsf}
\DeclareOldFontCommand{\tt}{\normalfont\ttfamily}{\mathtt}
\DeclareOldFontCommand{\bf}{\normalfont\bfseries}{\mathbf}
\DeclareOldFontCommand{\it}{\normalfont\itshape}{\mathit}
\DeclareOldFontCommand{\sl}{\normalfont\slshape}{\@nomath\sl}
\DeclareOldFontCommand{\sc}{\normalfont\scshape}{\@nomath\sc}
\newcommand{\nicefrac}[2]{\leave§ vmode\kern.1em
\raise.5ex\hbox{\the\scriptfont0 #1}\kern-.1em
/\kern-.15em\lower.25ex\hbox{\the\scriptfont0 #2}}

\newcommand{\C}{{\mathbb C}}

\newcommand{\R}{{\mathbb R}} 
\newcommand{\Z}{{\mathbb Z}}
\newcommand{\norm}[1]{\left \Vert {#1} \right \Vert}

\newcommand{\abs}[1]{{\left| {#1} \right|}}

\newcommand{\ZZ}{\mathcal{Z}}
\renewcommand{\Re}{\operatorname{Re}}

\nonstopmode
\begin{document}
\maketitle
\begin{abstract}
 Let $f$ be a zero-free analytic function on $\Re(s) \geq 1$. We prove that there exists an entire zero-free function $g$ 
 and a Helson zeta-function $\zeta_\chi(s)=\sum_{n=1}^\infty \chi(n) n^{-s}$, where $\chi(n)$ is a completely multiplicative unimodular function such that $f(s)=g(s) \zeta_\chi(s)$ for $\Re(s)>1$. By the Mittag-Leffler theorem  this implies that a Helson zeta-function  may  have meromorphic continuation from $\Re(s)>1$ to the complex plane with a prescribed set of zeros and poles in the half plane $\Re(s)<1$. This improves on results of Seip and Bochkov-Romanov who proved the same result in the strip $21/40<\Re(s)<1$ and conditional on the Riemann hypothesis in the strip $1/2< \Re(s)<1$. 
 Our results also gives information on maximum domains of meromorphicity and analyticity of Helson zeta-functions and show that any open connected set $U$ that includes the half plane $\Re(s) >1$, may be a maximum domain of meromorphicity or of analyticity for a Helson zeta-function. This extends results of Bhowmik and Schlage-Puchta to Dirichlet series with Euler products.
\end{abstract}

A Helson zeta-function $\zeta_\chi(s)$ is given by a Dirichlet series
\begin{gather*} 
  \zeta_\chi(s)=\sum_{n=1}^\infty \frac{\chi(n)} {n^s}=\prod_{p \text{ prime}} \left(1-\frac{\chi(p)}{p^{s}} \right)^{-1}, \qquad (\Re(s)>1)
\end{gather*}
where $\chi$ is a completely multiplicative
 unimodular\footnote{such that $|\chi(n)|=1$.} function. It is clear that $\zeta_\chi(s)$ is analytic for $\Re(s)>1$.  We study the possible analytic and meromorphic continuation to $\Re(s)<1$. 
 Our next theorem which gives all of our main results 
 depends on Lemma \ref{lem1} and 
 Lemma \ref{lem7} which will be stated and proven later.
 \begin{thm} \label{thm1}
  Let $f$ be a continuous function on $\Re(s) \geq 1$ which is analytic for $\Re(s) >1$.
Then there exist an entire function $g$ %
and a Helson zeta-function such that
$$
 f(s)=\frac{\zeta'_\chi(s)}{\zeta_\chi(s)}+g(s), \qquad (\Re(s)>1).
$$
\end{thm}

\begin{proof}
By Lemma \ref{lem1} there exists some entire function $g_0(s)$ and function  $f_0(s)$ such that 
\begin{gather} \label{fsdef}
 f(s)= f_0(s)+g_0(s), \\ \intertext{where} f_0(s)=\int_1^\infty
  q(x) x^{-s} dx, \notag
\end{gather}
 and $q(x)=o(1)$. Since $\chi$ is completely multiplicative it follows by taking the logarithmic derivative of its Euler product that
\begin{gather*} 
 \frac {\zeta_\chi'(s)} {\zeta_\chi(s)} 
  = - \sum_{n=1}^\infty \frac{\chi(n) \Lambda(n)} {n^{s}}, \qquad (\Re(s)>1), \\ \intertext{where}
  \Lambda(n)=\begin{cases} \log p & n=p^k \text { is a prime power}, \\ 0 & \text{otherwise}, \end{cases}
\end{gather*}
is the Van Mangoldt function. We proceed roughly as in Bochkov-Romanov \cite{BoRo}. Since we will be working also on $\Re(s) \leq 1/2$ where the sum over the prime powers will not be absolutely convergent we have to work with the Van Mangoldt function instead of just the primes. For $\Re(s)>1$ we get that
\begin{gather}  \label{bp}
 f_0(s)-\frac{\zeta'_\chi(s)}{\zeta_\chi(s)}=\int_1^\infty x^{-s} dr_1(x)=s \int_1^\infty r_1(x) x^{-s-1} dx, \\ \intertext{where}  \label{r1def}
 r_1(x):=\int_1^x q(t) dt + \sum_{n \leq x} \chi(n) \Lambda(n).
\end{gather}
Since $g_0(s)$ is entire it is sufficient to prove that $f_0(s)-\zeta_\chi'(s)/\zeta_\chi(s)$ has an analytic continuation to the entire complex plane. Our new idea is to continue to integrate \eqref{bp} by parts. We get for $\Re(s)>1$ that
\begin{gather} \label{iu}
  f_0(s)-\frac{\zeta_\chi'(s)}{\zeta_\chi(s)}=s(s+1) \cdots (s+j-1) \int_1^\infty r_j(x) x^{-s-j}dx,  \\ \intertext{where} \label{rkdef}
r_j(x):=\int_1^x r_{j-1}(t) dt, \qquad (j \geq 2). 
  \end{gather}
By Lemma \ref{lem7}  we may choose some completely multiplicative unimodular function $\chi$  and some $0<\theta<1$ such that
\begin{gather*} 
  \abs{r_j(x)} \ll_j x^{j\theta},
\end{gather*}
for each $j \geq 1$. For a given $j$  this estimate show that the  representation \eqref{iu} is absolutely convergent for and gives us an  analytic continuation of $g(s)$ to $\Re(s)>1-j(1-\theta)$. Since this holds for any $j \geq 1$,  the function $f_0(s)-\zeta'_\chi(s)/\zeta_\chi(s)$ may be analytically continued to an entire function $g_1(s)$ on the complex plane and our result follows by \eqref{fsdef} with $g(s)=g_0(s)+g_1(s)$. 
\end{proof}

  An immediate consequence of Theorem \ref{thm1} is the following result. 
  \begin{thm} \label{thm2}
  Let $f$ be an analytic zero-free function for $\Re(s)  > 1$, such that $f'(s)/f(s)$ extends to a continuous function on $\Re(s) \geq 1$. Then there exists some entire zero-free  function $g$ and some Helson zeta-function $\zeta_\chi$ such that
  \begin{gather*} 
    f(s)=\zeta_\chi(s) g(s),
    \end{gather*}
   for $\Re(s) > 1$.
\end{thm}
\begin{proof}
  Since $f$ is an analytic zero-free function for $\Re(s) \geq 1$ then $f'(s)/f(s)$
is analytic for $\Re(s) > 1$. Since it extends to a continuous function on $\Re(s)   \geq 1$, then by Theorem \ref{thm1} we have that there exists some entire function $h(s)$ and some Helson zeta-function $\zeta_\chi(s)$ such that
\begin{gather*}
   \frac{f'(s)}{f(s)} =\frac{\zeta_\chi'(s)}{\zeta_\chi(s)}+h(s), \qquad (\Re(s)>1).
\end{gather*}
By taking an antiderivative of this we get that
\begin{gather*}
   \log f(s) =\log(\zeta_\chi(s))+H(s), \qquad (\Re(s)>1),
\end{gather*}
 where  $H$ is an entire function. By exponenting both sides of the equation, we get that
$$
 f(s)=\zeta_\chi(s) g(s), \qquad (\Re(s)>1),
$$
where $g(s)=e^{H(s)}$ is an entire zero-free function.
\end{proof}
 By the Mittag-Leffler Theorem  \cite{MittagLeffler2} and Theorem  \ref{thm1} we obtain %
 \begin{thm} \label{thm3}
  Let  $U$ be an open connected set in $\C$ containing the half plane $\Re(s) >  1$, let $E \subset U \cap \{s: \Re(s)<1\}$ be a set without limit points on $U \cup(1+i \R)$, and let $\{g_a: a \in E\}$ be a set of  entire functions not identically zero but such that $g_a(0)=0$. Then there exists some Helson zeta-function $\zeta_\chi(s)$ which has analytic continuation to $U \setminus E$ such that the singular part of $\zeta_\chi'(s)/\zeta_\chi(s)$ at $a \in E$ is given by $g_a(\frac 1 {s-a})$, 
such that $\zeta_\chi(s)$ has continuous extension to $\Re(s) \geq 1$ and such that the maximal domain of analyticity 
 of $\zeta_\chi(s)$ is $U \setminus E$. 
 \end{thm}
 \begin{proof} 
 Since $E$ has no limit points on $1+i\R$ we may let $U_1$ be some neighbourhood of  $1+i\R$ such that $E \cap U_1=\emptyset.$ Let $U_2=U \cup U_1$.  By the Mittag-Leffler Theorem  \cite{MittagLeffler2}  where we in addition to poles also allow certain essential singularities\footnote{This was proved in Mittag-Leffler's paper, although in text-books it is most often stated just for meromorphic functions.}, we  construct a function $g$ which is analytic on $U_2 \setminus E$ with the given singular parts $g_a(1/(s-a))$ for  $a \in E$. We remark that since $g$ is analytic on $\Re(s) \geq 1$ it is in particular continuous on  $\Re(s) \geq 1$.
  Let $K_1=(1+i\R) \setminus U$, 
   let $\{a_1,a_2,\ldots\}$ be a dense set in $K_1$ and let
 $$
  h_1(s)=\sum_{n=1}^\infty 2^{-n } \left( 2 \exp \left( -\frac 1 {s-a_n}\right) -\exp \left( -\frac 2 {s-a_n}\right) \right).
 $$
Since the residues of $h_1(s)$ at $s=a_n$ are $0$ we may find an 
antiderivative $H_1(s)$ of $h_1(s)$ such that  $H_1$ is analytic on $U \cup (1+i\R)^\complement$ 
and thus also on $U$. While  $h_1(s)$ is discontinuous on $K_1$, by the triangle inequality it is bounded $|h_1(s)| \leq 3$ for $\Re(s) > 1$, 
 and its antiderivative $H_1(s)$ extends from  $\Re(s)>1$ to a continuous 
 function on  $ \Re(s) \geq 1$ such that  $H_1(s)$ is
  not analytic on any point in  $K_1$. 
Finally let $K_2=\partial U \setminus  ((1+i\R) \cup E)$. By \cite[Lemma 1]{BhoS} 
we may find a dense set
$\{b_1,b_2,\ldots\}$ in $K_2$ of points that are pathwise reachable in $U$. Let 
$$
h_2(s)=\sum_{n=1}^\infty \frac{1-\Re(b_n)} {2^n(s-b_n)}.%
$$ It is clear that $h_2$ is analytic on $\C \setminus K_2$ but not analytic on any point of $K_2$.
Thus the function
$$f(s)=g(s)+H_1(s)+h_2(s)$$
has the given singular parts on $E$, is analytic on $U$, has a continous extension to $\Re(s) \geq 1$  and is not analytic on any point of
 $\partial U=K_1 \cup K_2$ and thus have $U \setminus E$ as maximum domain of analyticity. Our proof is concluded by applying Theorem \ref{thm1} on the function $f$.
 \end{proof}
 \begin{thm} \label{thm4}
 Let $U \subseteq \C$ be an open connected set containing the half plane $\Re(s) > 1$, let $E \subset U$ be a set without limit points on $U \cup(1+i \R)$, let $\{g_a: a \in E\}$  be a set of entire functions which are zero-free for $s \neq 0$. Then there exists a completely multiplicative unimodular function $\chi$ such that the Helson zeta-function $\zeta_\chi(s)$ is analytic on $U$ and such that the limits
 $$
   \lim_{s \to a} \frac{\zeta_\chi(s)}{g_a(1/(s-a))}= c_a \neq 0,  \qquad (a \in E),
 $$
 exists.  Furthermore we may choose $\chi$ such that the maximal domain
  of analyticity of $\zeta_\chi(s)$ is  $U \setminus E$.
\end{thm}
\begin{proof}
 This follows from Theorem \ref{thm3}.
\end{proof}
 An immediate  consequence of Theorem \ref{thm4} is the following corrollary 
 \begin{cor} \label{cor1}
  Let $\{s:\Re(s)>1\} \subseteq U$ be an open connected set. Then there exists some Helson zeta-function $\zeta_\chi$ such that $\zeta_\chi$ may be analytically extended to a zero-free holomorphic function on  $U$ and furthermore $U$ is its maximum domain of holomorphicity and meromorphicity.
  \end{cor}
  This result  improve on results of Bhowmik and Schlage-Puchta \cite{BhoS}. They \cite[Theorem 1]{BhoS} showed the existence of a Dirichlet series $\sum_{n=1}^\infty a_n n^{-s}$ having any given open connected set including a half-plane $\Re(s) \geq \sigma_0$ as domain of meromorphicity. In \cite[Theorem 2, Theorem 3]{BhoS} they prove similar existence results for Dirichlet-series with an Euler-product. In particular our Corollary  \ref{cor1} implies \cite[Theorem 3]{BhoS}, without assuming the Riemann hypothesis and allowing $\Re(s) \leq 1$ rather than the strip $1/2 \leq \Re(s) \leq 1$. A particular case of Corollary \ref{cor1} is if $U$ is the half plane $\Re(s)>1$ then
 \begin{cor} \label{cor2}
  There exist some Helson zeta-function $\zeta_\chi$ such that $\zeta_\chi$ can not be analytically extended beyond the half plane $\Re(s)>1$.
  \end{cor}
  Similarly if $U=\C$ we obtain from Corollary \ref{cor1} that
  \begin{cor} \label{cor3}
  There exist some entire zero-free   Helson zeta-function $\zeta_\chi$.
 \end{cor}
 This corollary is originally due to Fedor Nazarov (fedja) who sketched a proof of this result on mathoverflow \cite{fedja} in 2010.
 Corollary \ref{cor3} can be compared with the fact that if $|\chi(p)|=1$ are chosen independently, uniformly and randomly, then almost surely the maximal domain of analyticity will be the half plane $\Re(s)>1/2$ (Saksman-Webb \cite{SaWe}, a similar result is also due to Queffelec \cite{Que}, see discussion in Bhowmik-Matsumoto \cite{BhoMat}). In view of the fact that a random Helson zeta-function is almost surely analytic, zero-free and pole free for $\Re(s)>1/2$ (Helson \cite{Helson}), and the fact that Seip \cite{Seip} and Bochkov-Romanov \cite{BoRo} shows that this need not be true, this might not be too surprising. We have a lot of freedom in constructing the multiplicative function  $\chi$ on the primes, as it can be done independently and there are a lot of primes.

In the special case where $g_a(s)=s^n$ for an integer $n$ in Theorem \ref{thm4} we get the existence of a  meromorphic Helson zeta-function with prescribed poles and zeros in a domain.
\begin{thm} \label{thm5}
 Let $U \subseteq \C$ be an open connected set containing the half plane $\Re(s) > 1$ and let $\ZZ \subset U \cap \{s: \Re(s)<1\}$ be any signed multiset without limit points on $U \cup(1+i\R)$. Then there exists a completely multiplicative unimodular function $\chi$ such that the Helson zeta-function $\zeta_\chi(s)$ has meromorphic continuation from $\Re(s)>1$ to $U$ with prescribed poles and zeros with given multiplicites from $\ZZ$. Furthermore we may choose $\chi$ such that the maximal domain of meromorphicity of $f$ is  $U$.
\end{thm} 
This improves on results of Seip \cite{Seip} who proved 
the same result for analytic functions with prescribed zeros in the strip $21/40<\Re(s)<1$
and results of Bochkov-Romanov \cite{BoRo} who also allow poles. 
An advantage with our method is that it gives us a result in the full complex plane. 
While Seip and Bochkov-Romanov proved that the strip may be extended to $1/2<\Re(s)<1$ conditional on the Riemann hypothesis, our result is valid unconditionally, and it gives prescibed poles and zeros on the full half plane $\Re(s)<1$. Before we prove the lemmas needed for our results we mention some possible further developments.
 As in Bochkov \cite{Bochkov} the proofs in our paper can be modified so that the results follows if  we 
 assume that the values of $\chi(p)$ are chosen from some finite set such that 0 is in the interior of the convex hull of such a set (such as cubic unit roots). If $\chi(p) \in \{-1,1\}$ similar results follows as long as $E$ and $U$ are symmetric with respect to the real axis. 
 It is also possible to prove our results for Beurling-Helson zeta-functions, 
 i.e. Dirichlet series over Beurling generalised integers with completely multiplicative unimodular coefficients, assuming some sufficient prime number theorem for Beurling primes in short intervals\footnote{By using some ideas in fedja \cite{fedja} it seems that it is possible to assume a weaker result on Beurling primes in short intervals than the "Hoheisel type" we use in Lemma \ref{hoh}}.
By assuming some conditions on the set of prescribed zeros, Seip \cite{Seip} also 
proved that $\zeta_\chi(s)$ can be chosen to be universal (in the Voronin sense) in 
some strip. We can weaken the conditions of Seip and also prove some universality results for our Helson zeta-functions. In particular we do not need the Helson zeta-function to satisfy a 
density hypthesis of the usual kind, and we need neither the Helson zeta-function
 nor its logarithm to satisfy a bounded mean square on some line $\Re(s)=\sigma<1$. 
 This will give us the first known Voronin universality result of this kind. We will do this in future work. Another interesting question is whether we can obtain some analogue of our results in several complex variables. Also here we have some ideas and will hopefully treat this in future work.

We now proceed to prove the lemmas needed to prove our main results.

\begin{lem} \label{lem1}
  Let $f$ be a continuous function on $\Re(s) \geq 1$ which is analytic for $\Re(s)>1$. Then there exists an entire function $g$ and a function $q(x)=o(1)$ such that 
   $$ f(s)=\int_1^\infty  q(x) x^{-s} dx +g(s),  \qquad (\Re(s) >1). $$
 \end{lem}

\begin{proof}
 By applying the Arakelyan approximation theorem  on the function $s^2 f(s)$ and the half-plane $\Re(s) \geq 1$  there exists some entire function $g_1(s)$ such that
 \begin{gather*}
    \abs{g_1(s)-  s^2 f(s)} \leq 1, \qquad (\Re(s) \geq 1).
 \end{gather*}
 By dividing this inequality by $|s|^2$ we get that 
 \begin{gather} \label{abaq}
    \abs{f_1(s)} \leq |s|^{-2}, \qquad (\Re(s) \geq 1), \intertext{where}
 f_1(s)= f(s)-g_1'(0)s^{-1}-g_1(0) s^{-2}  -g(s),  \, \, \, \, \, \text{and} \, \, \, \, \, g(s)=(g_1(s)-g_1(0)-g_1'(0)s) s^{-2} \, \, \, \, \,  \label{aaa1}\end{gather}
  is an entire function.
  It follows by  \eqref{abaq} and \cite[Lemma 2]{BoRo}\footnote{ Lemma 2 in \cite{BoRo} follows  by a Hardy class argument and a Payley-Wiener theorem.} 
 that
   \begin{gather} 
   f_1(s)=\int_{1}^\infty q_1(x) x^{-s} dx, \qquad q_1(x)=o(1), \label{aaa2}
  \end{gather}
 for $\Re(s)>1$. 
 Since $\int_1^\infty x^{-s-1}dx=s^{-1}$ and $\int_1^\infty (\log x) x^{-s-1} dx=s^{-2}$ it follows´from \eqref{aaa1} and \eqref{aaa2} that 
 \begin{gather*}
  f(s)= \int_1^\infty q(x) x^{-s} dx + g(s),   \intertext{where}  q(x)=q_1(x)+(g_1'(0)+g_1(0) \log x)x^{-1}=o(1).
\end{gather*}
\end{proof}

We  now proceed to prove Lemma \ref{lem7}, which is our main Lemma needed to prove our Theorems. In order to prove this result we need more Lemmas.

\begin{lem} \label{hoh}
There exist some $\theta<1$ such that for some constant $c>0$, some function $\delta:\R_{>0} \to \R_{>0}$ such that $\lim_{x \to \infty} \delta(x)=0$ and for $0 \leq a < b \leq 1$, $x \geq 2$ then
\begin{gather} \label{theta}
 \pi(x+bx^\theta)-\pi(x+ax^\theta) \geq  c(b-a-\delta(x)) \frac{x^\theta} {\log x}
\end{gather}
\end{lem}
\begin{proof}
 Hoheisel \cite{Hoheisel} gave the first proof that for some $\beta<1$ then
\begin{gather} \label{prime}
 \pi(x+x^\beta)- \pi(x) \gg \frac{x^\beta} {\log x}.
\end{gather}
More precisely he proved that \eqref{prime} holds for any $1-1/33000<\beta<1$.
Our lemma follows immediately for any $\beta<\theta<1$. We can also use some more modern result. For example from the result of Heath-Brown \cite{HeBr} we may choose $c=1$ and $\theta=7/12$ and by the result of Baker-Harman-Pintz \cite{BaHaPi} we may choose any $21/40<\theta<1$ for some $c>0$.
\end{proof}

\begin{lem} \label{lem3}
  Let $x_1,\ldots,x_n \in \C^d$ be vectors and $a_1,\ldots,a_n$ be complex numbers such that $|a_j| \leq 1$. Then there exist complex numbers $b_1,\ldots,b_n$ such that $|b_j|=1$ and
 $$\norm{\sum_{j=1}^n b_j x_j-\sum_{j=1}^n a_j x_j}_\infty \leq d \max_{1 \leq j \leq n} \norm{x_j}_\infty.$$
\end{lem}
\begin{proof}
 We may use Lemma 4.1 in Bagchi \cite{Bagchi}  recursively\footnote{The result works in any norm.}  until the remaining vectors are less than or equal to the dimension $d$.
\end{proof}               
\begin{lem} \label{lem4}
  Let $c>0$ and  $\theta<1$ be admissible in Lemma \ref{hoh} and let $w(x)=o(1)$ and $J \in \Z^+$. Then there exist some $X_0(J)$ such that if $x \geq X_0(J)$ and $|\xi_j| \leq w(x)  x^\theta$ for $1 \leq j \leq J$   the following linear system of equations has a solution
\begin{gather}
 \sum_{\substack{x<p<x+x^\theta \\ p \text{ prime}}} a_p \left(\frac{x+x^\theta-p}{x^\theta} \right)^{j-1} \log p =\xi_j, \qquad (j=1,\ldots,J),
\end{gather}
with $|a_p| \leq 1$.
\end{lem}
\begin{proof}
  Let $0<\varepsilon<1/(2(J+1))$. By Lemma \ref{hoh} we may choose a subset $\mathcal P$ of the primes such that for $0 \leq a<b \leq 1$ then the number of primes in each interval
$[x+a x^\theta,x+b x^\theta]$  is asymptotic to $c(b-a) x^\theta/\log x$. Now let $a_p= y_{m}$ if $\, \abs{p-x- \frac{J+1-m}{J+1} x^\theta}<\varepsilon$ and $p \in \mathcal P$ for $m=1,\ldots,J$ and let $a_p=0$ otherwise. %
For a sufficiently large $x$ we get that
\begin{gather}
  \sum_{\substack{x\leq p<x+x^\theta \\ p \text{ prime}}} a_p \left(\frac{x+x^\theta-p}{x^\theta} \right)^{j-1} \log p = 
\frac {2 c\varepsilon x^\theta}{(J+1)^{j}} \sum_{m=1}^{J}  m^{j-1} (1+ \varepsilon_m) y_m, \qquad (j=1,\ldots,J)
\end{gather}
where $\abs{\varepsilon_m} \leq 2\varepsilon$. We thus consider the linear system of equations
\begin{gather}
  \sum_{m=1}^{J} m^{j-1} (1+\varepsilon_m) y_m=\frac{\xi_j (J+1)^{j}}{2 c \varepsilon x^\theta}, \qquad (j=1,\ldots,J)
\end{gather}
When $\varepsilon_m=0$ the coefficient matrix is a Vandermonde matrix with non-zero determinant. Since the determinant of a matrix is a continuous function of its entries we may choose $\varepsilon$ sufficiently small so that when $|\varepsilon_m| \leq 2\varepsilon$ then  the absolute value of the determin                                      ant of the coefficient matrix is greater than $\delta$ for some $\delta>0$. We use the Cramer rule to solve the system of equations. Since $|\xi_j| \leq w(x) x^\theta$, the determinant in the numerator is bounded by $C_J  w(x)/(2c \varepsilon)$, when $|\varepsilon_m| \leq 1$ for $m=1,\ldots,J$ and where $C_J$ is an absolute positive constant only depending on $J$. We thus get that
\begin{gather}
 \max_{1 \leq m \leq J} \abs{y_m} \leq \frac{C_J w(x)}{2 c \varepsilon \delta},
\end{gather}
 Since $w(x)=o(1)$ the right hand side in the above inequality tend to zero as $x \to \infty$ and this gives us that for a sufficiently large $x \geq X_0(J)$ we have solutions to the system of equations with $|y_m| \leq 1$ and thus also  $|a_p|\leq 1$.
\end{proof}
By combining Lemma \ref{lem3} and Lemma \ref{lem4} and multiplying the inequality by $x^{(j-1)\theta}$ we arrive at the following result
\begin{lem} \label{lem5}
  Let $c>0$ and  $\theta<1$ be admissible in Lemma \ref{hoh} and let $w(x)=o(1)$. Then for any $J \in \Z^+$ there exist some $X_0(J)$ such that if $x \geq X_0(J)$ and $|\xi_j| \leq w(x)x^{j\theta}$ for $1 \leq j \leq J$, then there exist $|\chi(p)|=1$ for the primes $x \leq p<x+x^{\theta}$ such that   
\begin{gather}
  \abs{\sum_{\substack{x \leq p<x+x^\theta \\ p \text{ prime}}} \chi(p) \left(x+x^\theta-p \right)^{j-1} \log p - \xi_j} \leq J x^{(j-1)\theta}\log(x+x^\theta), \qquad (1 \leq j \leq J)
\end{gather}
\end{lem}
\begin{lem} \label{lem6}
 Let $0<\theta<1$ be admissible in Lemma \ref{hoh}, let $x_1=1$ and $x_{k+1}=x_k+x_k^\theta$. Suppose  that for some $J \geq 1$  and for some sufficiently large $k=k_0$ that  the completely multiplicative unimodular function $\chi(p)$ is defined on the primes $p < x_{k_0}$  such that with $r_j(x)$ defined by \eqref{r1def} and \eqref{rkdef} then we have with $k=k_0$ that
   \begin{gather} \abs{r_j(x_k)} \leq \frac{(J+1)}{(j-1)!} x_{k}^{(j-1) \theta}\log x_{k} \qquad (1 \leq j \leq J). \label{iii}
   \end{gather}
   Then there exist some arbitrarily large $k_1>k_0$ and definition of the unimodular completely multiplicative function $\chi(p)$ on the primes $x_{k_0} \leq p < x_{k_1}$ such that \eqref{iii} holds for $k_0  \leq k<k_1$ and such that
  \begin{gather} \label{iv}    \abs{r_j(x_{k_1})} \leq \frac{(J+2)}{(j-1)!} x_{k_1}^{(j-1) \theta}\log x_{k_1} \qquad (1 \leq j \leq J+1)
   \end{gather}
\end{lem}
\begin{proof}
  By writing $r_j(x)$ as a Taylor polynomial with integral remainder term we get
\begin{gather} \label{Taylor1}  r_j(x)  = P_j(x)+R_j(x), \qquad (x_k \leq x  \leq x_{k+1}) \\ \intertext{where} \label{Taylor2} P_j(x)=\sum_{m=0}^{j-1} r_{j-m}(x_k) \frac{(x-x_k)^m}{m!},  \qquad \text{and} \qquad R_j(x)=\frac 1 {(j-1)!} \int_{x_k}^x (x-t)^{j-1}  dr_1(t). \qquad
\end{gather}
We may write 
\begin{gather}  \label{rem1}
 R_j(x)=\frac 1 {(j-1)!} \sum_{x_k \leq  p<x }  \chi(p) (x-p)^{j-1} \log p+E_j(x), \\ \intertext{where} \label{rem2} E_j(x)=\frac 1 {(j-1)!} \int_{x_k}^{x} (x-t)^{j-1} q(t) dt +\frac 1 {(j-1)!} \sum_{\substack{x_k \leq p^v<x \\ p \text{ prime}, v \geq 2} }  \chi(p^v)(x-p^v)^{j-1} \log p  \qquad
\end{gather}
We  now construct the completely multiplicative function recursively for $x_k \leq p<x_{k+1}$ starting with $k=k_0$ such that \eqref{iii} hold at each step. By the fact that $q(x) =o(1)$ and the fact that the prime powers $p^v$ with $v \geq 2$ are sparse in  short intervals\footnote{less than the number of powers of integers $n^v$ for $v \geq 2$ is sufficient} it follows    that
\begin{gather*}
 \abs{E_j(x_{k+1})}= o(x^{j \theta}), \qquad (1\leq j \leq J+1) \label{o1}
\end{gather*}
By our assumption \eqref{iii} it furthermore follows by \eqref{Taylor2} and the triangle inequality that 
\begin{gather*}
 \abs{P_j(x_{k+1})}\ll_J x_k^{(j-1)\theta} \log x_k \qquad (1 \leq j \leq J)  \label{o2}
\end{gather*}
We arrive  at
\begin{gather} \label{jir}
  r_j(x_{k+1})=\frac 1 {(j-1)!} \sum_{x_k \leq  p<x_{k+1} }  \chi(p) (x_{k+1}-p)^{j-1} \log p+ \xi_j,  \qquad (1\leq j \leq J)
\end{gather}
where  $|\xi_j|=o(x^{j\theta}), $
For $j=J+1$ we can prove something similar, but we do not have an estimate for $r_{J+1}(x_k)$. However proceeding
in the same way we get that
\begin{gather*}
 r_{J+1}(x_{k+1})-r_{J+1}(x_k) =\frac 1 {J!} \sum_{x_k \leq  p<x_{k+1} } 
 \chi(p) (x_{k+1}-p)^{J} \log p + \xi
\end{gather*}
where $|\xi|=o(x^{(J+1)\theta})$. We now define 
\begin{gather*}
  \xi_{J+1}:= \xi+\begin{cases}  -r_{J+1}(x_{k}) & |r_{J+1}(x_{k})| \leq  8J x_k^{J\theta} \log x_k   \\
 -4Jx^{J \theta} \log x_k \frac{r_{J+1}(x_k)}{\abs{r_{J+1}(x_k)}} & \text{otherwise}\end{cases}
\end{gather*}
We now use Lemma \ref{lem5} to define $\chi(p)$ for $x_k \leq p<x_{k+1}$
such that
\begin{gather*}
 \abs{\frac 1 {(j-1)!} \sum_{x_k \leq  p<x_{k+1} }  \chi(p) (x_{k+1}-p)^{j-1} \log p+ \xi_j} \leq\frac{ (J+1)}{(j-1)!} x_k^{(j-1)  \theta} \log x_k,  \qquad (1\leq j \leq J+1).
\end{gather*}
It is clear by \eqref{jir} that \eqref{iii} also holds for $x_{k+1}$. At each step in the recursion process $|r_{J+1}(x_{k})|$ will be smaller until it is small enough for \eqref{iv} to hold. In particular it follows that \eqref{iv} holds for any $k_1 \geq k_0+|r_{J+1}(x_{k_0})/(2J \log x_{k_0})|$.
\end{proof}

 \begin{lem} \label{lem7}
  Let $0<\theta<1$ be admissible in Lemma \ref{hoh}. Then there exist some completely multiplicative unimodular function $\chi$ such that with $r_j(x)$ defined by \eqref{r1def} and \eqref{rkdef} where $q(x)=o(1)$, then
 \begin{gather} \label{ajaj}
  \abs{r_j(x)} \ll_j x^{j \theta}.
\end{gather}
\end{lem}
\begin{proof}
 As in Lemma \ref{lem6} we will let $x_1=1$, $x_{k+1}=x_k+x_k^\theta$. Choose some sufficiently large $k_0<k_1$ such that $  3/2 \leq x_{k_1}/{x_{k_0}} \leq 2$. Let $p_1<\dots<p_M$ be an enumeration of the primes less than $x_{k_0}$ and let $\chi(p_m)=(-1)^m$ for $1 \leq m \leq M$. Now let
\begin{gather} \label{eq2a} A=\int_1^{x_{k_1}} q(x) dx +\sum_{1 \leq n < x_{k_0}} \chi(n)\Lambda(n) +
\sum_{\substack{x_{k_0} \leq p^v <x_{k_1}\\ p \text{ prime}, v \geq 2}} \chi(p^v)\log p \end{gather}
By the definition of $\chi(p)$ for $p \leq x_{k_1}$, the fact that prime powers of higher order than $1$ are sufficiently sparse and the fact that $q(x)=o(1)$ it follows that $\abs{A} =o(x_{k_1})$ and since by the prime number theorem there are $\gg x_{k_1}/\log x_{k_1}$ primes in the interval $[x_{k_0},x_{k_1})$ we may define $|\chi(p)|=1$ in this interval such that 
\begin{gather} \label{eq2b}
  \Big | \sum_{\substack{x_{k_0} \leq p <x_{k_1} \\ p \text{ prime}} } \chi(p) \log p+A \, \Big | \leq \log x_{k_1}.
\end{gather}
Thus by \eqref{r1def}, \eqref{eq2a} and \eqref{eq2b} we get
\begin{gather*}
  |r_1(x_{k_1})| \leq \log x_{k_1}.
 \end{gather*}
This allows us for $J \geq 1$ to use Lemma \ref{lem6} recursively to define $k_{J+1} >  \max(k_{J},2^J)$ and the completely multiplicative unimodular function $\chi(p)$ on the primes $x_{k_J} \leq p <x_{k_{J+1}}$ such that
\begin{gather} 
     \abs{r_j(x_k)} \leq
      \frac{(J+1)}{(j-1)!} x_{k}^{(j-1) \theta} \log x_{k} \qquad (1 \leq j \leq J, \, k_J \leq k \leq k_{J+1}).      \label{er1}
   \end{gather}
Thus \eqref{ajaj} holds for $x=x_k$. It remains to prove \eqref{ajaj} for $x_k \leq x \leq x_{k+1}$.  By estimating $r_j(x)$ in the interval $x_k 
<x<x_{k+1}$  by the Taylor polynomial with remainder term \eqref{Taylor1},\eqref{Taylor2}, where the polynomial term is estimated by \eqref{er1} and remainder term $R_j(x)$ is estimated by \eqref{rem1}, \eqref{rem2}
trivially\footnote{By sieve methods it is well known that  $\pi(x+x^\alpha)-\pi(x)  \ll x^\alpha/\log x$ for any $\alpha>0$.} by the triangle inequality and absolute values we find that the estimate \eqref{ajaj} hold also in the interval $x_k<x<x_{k+1}$.
\end{proof}

\bibliographystyle{plain}

\end{document}